\documentclass[12pt]{article}
\usepackage[utf8]{inputenc}
\usepackage[T1]{fontenc}
\usepackage[margin=1.2in]{geometry}
\usepackage{amsmath}
\usepackage{amssymb}
\usepackage{amsthm}
\usepackage{mathtools}
\usepackage{authblk}
\usepackage{hyperref}
\usepackage{doi}
\usepackage{tikz}
\usepackage{url}
\usepackage{graphicx}
\usepackage{enumerate}
\usepackage{appendix}
\usepackage[style=alphabetic, maxalphanames=4, maxbibnames=10]{biblatex}
\usepackage[capitalise]{cleveref}
\usepackage{etoolbox}
\cslet{blx@noerroretextools}\empty
\usepackage{biblatex}
\usepackage{autonum}

\bibliography{main}

\newtheorem{theorem}{Theorem}[subsection]

\newtheorem{lemma}[theorem]{Lemma}
\newtheorem{proposition}[theorem]{Proposition}
\newtheorem{definition}[theorem]{Definition}

\newcommand{\norm}[1]{\left\lVert #1 \right\rVert}

\numberwithin{equation}{section}
\numberwithin{theorem}{section}

\title{Basis-Free Analysis of Singular Tuples and Eigenpairs of Tensors}

\author[1]{Joao Marcos Vensi Basso\thanks{Now at Google Quantum AI}}
\author[2]{Loring W. Tu}
\affil[1,2]{Department of Mathematics, Tufts University}

\date{\today}

\newcommand{\R}{\mathbb{R}}
\newcommand{\comp}{\mathrel{\scriptstyle\circ}}
\newcommand{\dual}{^{\scriptstyle\vee}}
\newcommand{\Exterior}{\bigwedge\nolimits}
\newcommand{\Hom}{\operatorname{Hom}}
\newcommand{\Sym}{\operatorname{Sym}}
\DeclareMathOperator{\sgn}{sgn}

\begin{document}

\maketitle

\begin{abstract}
A tensor in applied mathematics is usually defined as a multidimensional array of numbers.  This presumes a choice of basis in $\R^n$ or in some other vector space, and tensorial concepts are defined accordingly.  In this article we define eigenvalues, eigenvectors, singular values, and singular vectors of a tensor intrinsically, without reference to a basis.  The basis-free approach has several advantages.  First, it shows more clearly the relationship between tensor analysis and areas of pure mathematics such as abstract algebra, differential topology, and algebraic geometry.  Second, it obviates the need to prove that a concept defined in terms of coordinates is independent of the choice of basis.  Third, an intrinsic definition is usually conceptually simpler.  As illustrations we show how Morse theory from differential topology can be used to analyze eigenvalues and eigenvectors of a symmetric tensor.  We also reprove a few results that are obvious in the basis-free approach, but not otherwise.
\end{abstract}

\section{Introduction}
The work to generalize matrix concepts such as eigenpairs to tensors started with Qi \cite{Qi05, Qi07} and Lim \cite{Lim05}. Both authors have different approaches but arrive at similar results. More work has been done since then, including but not limited to \cite{CS13}, \cite{SSZ13} and \cite{HHLQ13}. The book \cite{Lan12} is a survey of tensors from a geometric perspective, however it revolves around tensor rank and has no mention of eigenpairs. This serves as evidence of how new this area of research is.

Moreover, to our knowledge, all the work done so far has been in terms of coordinates. Although \cite{Qi05} shows that many of those concepts are independent of bases, they are still framed in a coordinate-dependent manner. In this work, we recast many of these tensor concepts in a coordinate-free language.

For conciseness, we refer to an eigenvalue and its corresponding eigenvector as an \textbf{eigenpair}. If the eigenvector has unit length, we will refer to the pair as the \textbf{unit eigenpair}. Similarly, we refer to a singular value and its corresponding singular vectors as a \textbf{singular tuple}.

In \cref{preliminaries_section} we recall a few concepts from pure mathematics. In \cref{new_defs_section} we present our coordinate-free definitions in the $L^2$ norm. In \cref{any_norm_section}, we generalize part of \cref{new_defs_section} to any norm. Finally, in \cref{morse_section}, we show a connection to Morse Theory that leads to constraints on eigenpairs of tensors. \cref{appendix_var_lemmas} collects a few lemmas used throughout the paper.

\section{Algebraic, Topological, and Geometric Preliminaries}\label{preliminaries_section}

We briefly review some coordinate-free notions from algebra, topology, and geometry that will be needed later.

Let $V$ and $W$ be real vector spaces.  One can construct the tensor product $V\otimes W$ in the usual way.
Denote by $V\dual$ the dual space of $V$, the vector space of all linear maps from $V$ to $\R$:

\begin{equation}
V\dual := \Hom(V, \R) = \{ \text{linear functions } f\colon V \to \R \}.
\end{equation}
 There are canonical isomorphisms
\begin{equation}
V\dual \otimes W\dual \simeq (V\otimes W)\dual = \Hom(V\otimes W, \R) \simeq \{ \text{bilinear maps}\colon V \times W \to \R\},
\end{equation}
the last isomorphism being the universal mapping property of the tensor product.

We are particularly interested in tensors in 
$V_1\dual \otimes \cdots \otimes V_k\dual$, $\dim V_i = n_i$.
These tensors are in a canonical one-to-one correspondence with $k$-linear maps on $V_1 \times \cdots \times V_k$, as follows.
Let $T$ be such an order-$k$ tensor.
Since there is a canonical isomorphism
\begin{equation}
V_1\dual \otimes \cdots \otimes V_k\dual \simeq (V_1 \otimes \cdots \otimes V_k)\dual,
\end{equation}
\noindent the tensor $T$ may be interpreted as a linear function from $V_1 \otimes \cdots \otimes V_k$ to $\R$.
By the universal mapping property, the linear map $T\colon V_1 \otimes \cdots \otimes V_k \to \R$ corresponds to a unique $k$-linear map $f_T\colon V_1 \times \cdots \times V_k \to \R$ such that
\begin{equation}
f_T(v_1, \ldots, v_k) = T(v_1 \otimes \cdots \otimes v_k).
\end{equation}
In terms of coordinates, if $T$ is the array $[T_{i_1\cdots i_k}]$, $1 \le i_j \le n_j$, then
\begin{equation}\label{associated_poly_eq}
f_T(x_1, \ldots, x_k) = \sum T_{i_1\cdots i_k} x_{i_1}^{(1)}\cdots x_{i_k}^{(k)},
\end{equation}
where the sum runs over all $1 \le i_j \le n_j$ and $\pmb{x}^{(i)} = (x_1^{(i)}, \ldots, x_n^{(k)})$ are vectors in $V_i$. In general, we will use $v \in V$ to represent an abstract vector and $\pmb{x}$ the vector after a choice of coordinates. If a tensor $T$ is square, we may call it a \textbf{order-$k$ tensor on $V$}, by which we mean an element of $(V\dual)^{\otimes k}$.  For a tensor $T$, we will denote its associated multilinear polynomial by $f_T$.

In a way similar to how the exterior power $\Exterior^k V$ is constructed (e.g. \cite[Section 19]{Tu17}), one can construct the symmetric power $\Sym^k V$ of a vector space.  
Elements of $\Sym^k(V\dual)$ correspond canonically to homogeneous polynomials of degree $k$ on $V$.
In terms of coordinates, if $[T_{i_1\cdots i_k}]$, $1 \le i_j \le n$, is an order-$k$ symmetric tensor on $V$, then $T$ corresponds to the homogeneous polynomial
\begin{equation}
f_T(\pmb{x}, \cdots, \pmb{x}) = \sum T_{i_1\cdots i_k} x_{i_1} \cdots x_{i_k},
\end{equation}
where the sum runs over all $1\le i_j \le n$ and $\pmb{x} = (x_1, \ldots, x_n)$.

Now suppose $M$ and $N$ are smooth manifolds.  Denote the tangent space to $M$ at a point $p$ by $T_pM$.
\begin{definition}
Let $f\colon N \to M$ be a smooth map and $p \in N$.
The \textbf{differential} $f_*\colon T_pN \to T_{f(p)}N$ of $f$ at $p$ is defined as follows.  
Let $X \in T_pN$ and $c(t)$ a curve in $N$ with initial point $c(0)=p$ and initial vector $c'(0) = X$.
Then $f_*(X) = (f\comp c)'(0)$.
\end{definition}

The differential of a smooth map is the coordinate-free generalization of the derivative.

\begin{definition}
Let $f\colon N \to M$ be a smooth map of manifolds.
If the differential $f_*\colon T_pN \to T_{f(p)}M$ is surjective, then $p$ is called a \textbf{regular point} of $f$; otherwise, $p$ is a \textbf{critical points}.
The image $f(p)$ of a critical point $p$ is called a \textbf{critical value}.
\end{definition}

Next suppose that $V$ is a finite-dimensional inner product space with inner product $\langle \ , \ \rangle$.
Then there is a canonical isomorphism $V \to V\dual$ given by $X \mapsto \langle X, \ \rangle$.
We denote its inverse $V\dual \to V$ by $(\ )^{\#}$.
Thus, if $\omega$ is a 1-form on a Riemannian manifold $M$, then $\omega^{\#}$ is a vector field on $M$.

\begin{definition}\label{gradient_def}
If $f\colon M \to \R$ is a function on a Riemannian manifold $M$, 
then its \textbf{gradient} $\nabla f$ is the vector field $(df)^{\#}$.
\end{definition}

If $f \colon V_1 \times V_2 \rightarrow \mathbb{R}$ and $v_i \in V_i$, we may write $\nabla_{v_1} f(v_1,v_2)$ or $\nabla_1 f(v_1,v_2)$ to signify the gradient of $f$ only with respect to $v_1$, that is, with $v_2$ kept constant.

\begin{proposition}\label{prop_extrema_critical_points}
If a smooth function $f \colon M \rightarrow \mathbb{R}$ on a $C^\infty$ manifold $M$ has a maximum or a minimum at $p \in M$, then $p$ is a critical point of $f$.
\end{proposition}

\begin{proof}
Let $v \in T_pM$. Choose a curve $c(t)$ in $M$ such that $c(0) = p$ and $c'(0) = v$. Using the curve $c(t)$ to compute the differential $f_{*,p}$ (\cite{Tu11}, Proposition 8.18, p.95), we have

\begin{equation}\label{manifold_equation_proof}
    f_{*,p}(v) = \frac{d}{dt}\Bigr|_{\substack{t=0}}f(c(t)) = (f \circ c)'(0)
\end{equation}

Note that $f \circ c \colon (-\epsilon, \epsilon) \rightarrow \mathbb{R}$ is a real-valued function that has a maximum or minimum at $0$. Therefore, $(f \circ c)'(0) = 0$. By \cref{manifold_equation_proof}, $f_{*,p} = 0$ for all $v \in T_pM$. Therefore, $p$ is a critical point of $f$.
\end{proof}

Note that all the concepts defined in this section are coordinate-free.  Therefore, anything defined in terms of them will also be coordinate-free.

\section{Basis-Free Definitions for the $L^2$-Norm}\label{new_defs_section}

Throughout this section, we use the $L^2$ norm and $V$ will denote a real inner product space of dimension $n$.

\subsection{Eigenpairs of a Symmetric Tensor}

In order to motivate our definition as in \cite{Lim05}, we mention a less common definition of eigenpairs of symmetric matrices.

\begin{definition}[Variational definition of matrix eigenpairs]\label{matrix_eigen_var_def}
If the symmetric matrix $T$ on $V$ has associated polynomial $f_T$, its \textbf{unit eigenvectors} are the critical points of $f_T$ on the unit hypersphere $S^{n-1}$ and the associated \textbf{eigenvalues} are the corresponding critical values.
\end{definition}

This can be proven to agree with the common definition of matrix eigenpairs by optimizing $f_T(\pmb{x})$ with the constraint $\norm{\pmb{x}}_2 =1$ with the method of Langrage multipliers:

\begin{equation}\label{eigen_langrangian_eq}
    L(\pmb{x}, \lambda) = f_T(\pmb{x}) - \lambda(\norm{\pmb{x}}^2 - 1).
\end{equation}

By an order-$k$ symmetric tensor $T$ on $V$, we mean an element of the 
symmetric power $\Sym^k V\dual$, 
which is in one-to-one correspondence with the vector space 
of homogeneous polynomials of degree $k$ on $V$. Inside the vector space $V$, there is a unit sphere 
$S^{n-1}$ of dimension $n-1$ centered at the origin. Hence, it is natural to define tensor eigenpairs as follows.

\begin{definition}[Coordinate-free eigenpairs of symmetric tensors]\label{sym_tensor_eigenpair_def}
Let $f$ be the homogeneous polynomial of degree $k$ on $V$ corresponding
to the symmetric order-$k$ tensor $T$ on $V$,
and $f|_{S^{n-1}}\colon S^{n-1} \to \R$ the restriction of $f$
to the unit sphere $S^{n-1}$. The \textbf{unit eigenvectors} of the symmetric tensor $T$ are the critical points of $f|_{S^{n-1}}$ and the associated \textbf{eigenvalues} are the corresponding critical values.
\end{definition}

This naturally agrees with the standard definition for matrices by \cref{matrix_eigen_var_def}. We can recast this definition in a language similar to that in \cite{Lim05}. For that, we enforce $\nabla L = \pmb{0}$ where, analogously to \cref{eigen_langrangian_eq}, $L(v, \lambda) = f_T(v,\cdots,v) - \lambda(\norm{v}^k-1)$. $\nabla_{\lambda} L(v,\lambda) = 0$ implies $\norm{v} = 1$. Using \cref{grad_norm_lemma}, $\nabla_v \norm{v} = v$ in the $L^2$ norm. Hence, $\nabla_v L(v,\lambda) =0$ implies that $\nabla_v f_T(v,\cdots,v) = k\lambda v$. Multiplying both sides by $v^\top$ and using \cref{euler_homo_lemma} since $f_T$ is homogeneous of order $k$, we get $\lambda = f(v,\cdots,v)$, the eigenvalue. This leads to \cref{eigen_sym_grad_def}.

\begin{definition}[Alternative to \cref{sym_tensor_eigenpair_def}]\label{eigen_sym_grad_def}
Let $f_T\colon V^k \to \R$ be the $k$-linear map corresponding to a order-$k$ symmetric tensor $T$ on $V$. Then $(v, \lambda) \in V \times \mathbb{R}$ is an \textbf{unit eigenpair} if

\begin{align}
    \nabla_1 f_T(v, \cdots, v) &= k \lambda v \\
    \norm{v}^k &= 1
\end{align}
\end{definition}

The choice of $\nabla_1$ is arbitrary since $f_T$ is symmetric and any other component yields the same result. Note that, by \cref{gradient_def}, this definition is basis-independent since our formulation of the gradient is basis-independent.

From this definition, an analogous result from symmetric matrices follows.

\begin{theorem}
If $T$ is a symmetric tensor, all its eigenvalues are real.
\end{theorem}

\begin{proof}
This follows from the fact that the eigenvalues are the critical values of $f_T$, which is real-valued.
\end{proof}

\subsection{Eigenpairs of a Square Tensor}

In order to define the eigenpairs of square (not necessarily symmetric) tensors,  we extend \cref{eigen_sym_grad_def}. However, since $f_T$ is not necessarily symmetric, each choice of $\nabla_i$ may result in a different eigenpair.

\begin{definition}[Coordinate-free eigenpairs of square tensors]\label{sq_tensor_eigenpair_def}
Let $f_T\colon V^k \to \R$ be the $k$-linear map corresponding to an order-$k$ square tensor $T$ on $V$. The pair
$(v,\lambda)\in V \times \R$ is an \textbf{mode-$i$ unit eigenpair} of $T$ if
\begin{align}
\nabla_i f_T(v, \ldots, v) &= k\lambda v \\
\norm{v}^k &= 1
\end{align}
\end{definition}

Upon a choice of basis, this definition agrees with that in \cite{Lim05}. Although this might appear to be a repetition of the results in \cite{Lim05}, note the crucial difference that now there is no reference to coordinates.

Relative to a basis, for $k=2$, $T$ is represented by a matrix $A$.  
In this case, there are two modes $1$ and $2$. 
Mode-$1$ eigenpairs of $T$ are the usual (right) eigenpairs of $A$,
and mode-$2$ eigenpairs of $T$ are the usual left eigenpairs of $A$, or equivalently, the usual right eigenpairs of $A^\top$.

\subsection{Singular Pairs of a Rectangular Tensor}

As before, we motivate our definitions by a variational definition of matrix singular values.

\begin{definition}[Variational definition of matrix singular values]\label{var_singular_matrix_def}
If the rectangular tensor on $V \times W$ has associated polynomial $f_T$, $\dim V = m$ and $\dim W = n$, its \textbf{singular vectors} are the critical points of $f_T$ on $S^{m-1} \times S^{n-1}$ and the associated \textbf{singular values} are the corresponding critical values.
\end{definition}

This can be proven to agree with the usual definition of matrix singular pairs by optimizing $f_T(\pmb{x}, \pmb{y})$ with the constraint $\norm{x} = 1 = \norm{y}$ through Lagrange multipliers:

\begin{equation}\label{singular_matrix_lagrangian_eq}
    L(\pmb{x}, \pmb{y}, \sigma', \sigma'') = f_T(\pmb{x}, \pmb{y}) - \sigma'(\norm{\pmb{x}} - 1) - \sigma''(\norm{\pmb{y}} - 1)
\end{equation}

To generalize to tensors, let $V_1, V_2, \ldots, V_k$ be real inner product spaces of dimensions $n_1, \ldots, n_k$ respectively.
A tensor on $V_1 \times  \cdots \times V_k$ is an element
of the tensor product $V_1\dual  \otimes \cdots \otimes V_k\dual$.
By the universal mapping property of the tensor product, such a tensor corresponds
to a $k$-linear functional on the Cartesian product
$V_1 \times \cdots \times V_k$.

\begin{definition}[Coordinate-free singular tuples of tensors]\label{coord_free_sing_def}
Let $f_T$ be the homogeneous polynomial corresponding to the tensor $T \in (V_1 \otimes \cdots \otimes V_k)\dual$. The \textbf{singular vectors} of the $T$ are the critical points of $f_T|_{S^{n_1-1} \times \cdots \times S^{n_k-1}}$ and the associated \textbf{singular values} are the corresponding critical values.
\end{definition}

This naturally agrees with the standard definition for matrices by \cref{var_singular_matrix_def}. We can again recast this definition in the language of \cite{Lim05}. Define a Langrangian analogous to \cref{singular_matrix_lagrangian_eq}:

\begin{equation}
    L(\pmb{x}_1,\cdots,\pmb{x}_k, \sigma^{(1)},\cdots,\sigma^{(k)}) = f_T(\pmb{x}_1,\cdots,\pmb{x}_k) - \sum_{j=1}^k \sigma^{(j)}(\norm{\pmb{x}_j}-1)
\end{equation}

For all $i$, enforcing $\nabla_{\sigma^{(i)}} L = 0$ yields $\norm{\pmb{x}_i} = 1$. Furthermore, $\nabla_{\pmb{x}_i} L = 0$ implies $\nabla_{\pmb{x}_i} f_T(\pmb{x}_1,\cdots,\pmb{x}_k) = \sigma^{(i)} \pmb{x}_i$ by \cref{grad_norm_lemma}. Multiplying both sides by $\pmb{x}_i^\top$ and using \cref{grad_f_T_lemma}, we obtain $\sigma^{(1)} = \cdots = \sigma^{(k)} = f_T(\pmb{x}_1,\cdots,\pmb{x}_k) := \sigma$, the singular value. This leads to \cref{singular_tensor_grad_def}.

\begin{definition}[Alternative to \cref{coord_free_sing_def}]\label{singular_tensor_grad_def}
Let $f_T\colon V_1 \times  \cdots \times V_k \to \R$ 
be the $k$-linear function corresponding to the tensor $T \in (V_1 \otimes \cdots \otimes V_k)\dual$.
$(v_1, \ldots, v_k, \sigma)\in V^1 \times \cdots \times v^k \times \R$ is a \textbf{singular tuple} of $T$ if
\begin{align}
\nabla_i f(v_1, \cdots, v_k) &= \sigma v_i \\
\norm{v_i} &= 1
\end{align}

\noindent for all $i \in [k]$.
\end{definition}

\section{Generalization to any norm}\label{any_norm_section}

In order to generalize singular tuples and eigenpairs to any norm, we do not refer to coordinates and do not evaluate $\nabla_{\pmb{x}_i} \norm{\pmb{x}_i}^k$.

\begin{definition}[Coordinate-free eigenpairs of square tensors for any norm]\label{sq_tensor_eigenpair_any_norm_def}
Let $f_T\colon V^k \to \R$ be the $k$-linear map corresponding to an order-$k$ square tensor $T$ on $V$. The pair
$(v,\lambda)\in V \times \R$ is an \textbf{mode-$i$ unit eigenpair} of $T$ if
\begin{align}
\nabla_i f_T(v, \ldots, v) &= \lambda \nabla_i \norm{v}^k \\
\norm{v}^k &= 1
\end{align}
\end{definition}

For singular values, we still assume an $L^p$ norm for any $p$ since we use \cref{grad_norm_lemma}.

\begin{definition}[Coordinate-free singular tuples of rectangular tensors in any norm]\label{singular_tensor_grad_any_norm_def}
Let $f_T\colon V_1 \times  \cdots \times V_k \to \R$ 
be the $k$-linear function corresponding to the tensor $T \in (V_1 \otimes \cdots \otimes V_k)\dual$.
$(v_1, \ldots, v_k, \sigma)\in V^1 \times \cdots \times v^k \times \R$ is a \textbf{singular tuple} of $T$ if
\begin{align}
\nabla_i f(v_1, \cdots, v_k) &= \sigma \nabla_i \norm{v_i} \\
\norm{v_i} &= 1
\end{align}

\noindent for all $i \in [k]$.
\end{definition}

\section{Eigenpairs constraints from Morse Theory}\label{morse_section}

\subsection{Morse theory overview}

Since our definition of eigenpairs are related to optimizating a function on $S^{n-1}$, we obtain restrictions on eigenpairs by using results in Morse theory \cite{Mil16}, linking tensors to algebraic topology. We mention a few preliminary results from Morse theory.

\begin{definition}[\cite{Mil16}]
Let $M$ be a manifold and $f \colon M \rightarrow \R$. A critical point $p$ of $f$ is \textbf{nondegenrate} if the Hessian $H(f)$ of $f$ is nonsingular at $p$.
\end{definition}

Recall that the Hessian of $f$ is its matrix of second partial derivatives:

\begin{equation}
    \Big[ H(f) \Big]_{i,j} := \frac{\partial^2 f}{\partial x_i \partial x_j}
\end{equation}

\noindent where the $x_i$ are local coordinates on $M$. By \cite{Mil16}, the Hessian is coordinate-invariant.

\begin{theorem}[\cite{Mil16}]\label{milnor_bijection_lemma}
If a smooth function $f \colon M \rightarrow R$ on a manifold $M$ has only nondegenerate critical points, then there is a cell decomposition of $M$ such that there exists a bijection

\begin{equation}
    \text{critical point of index } \lambda \leftrightarrow \text{cell of dimension } \lambda
\end{equation}
\end{theorem}

For conciseness, we use the following notation.

\begin{definition}
For a manifold $M$, let

\begin{align}
    a_\lambda &:= \text{ number of cells of dimension } \lambda\\
    b_\lambda &:= \dim H_\lambda(M) \\
    c_\lambda &:= \text{ number of critical points of index } \lambda
\end{align}
\end{definition}

\begin{lemma}[Weak Morse Inequality \cite{Mil16}]\label{weak_morse_inequality_lemma}
For a manifold $M$,
\begin{equation}
    b_\lambda \leq c_\lambda
\end{equation}
\end{lemma}

\begin{lemma}[Strong Morse Inequality \cite{Mil16}]\label{strong_morse_inequality_lemma}
For a manifold $M$,

\begin{equation}
    b_\lambda  - b_{\lambda-1} + b_{\lambda-2} - \cdots \pm b_0 \leq c_\lambda - c_{\lambda-1} + c_{\lambda-2} - \cdots \pm c_0
\end{equation}
\end{lemma}

\begin{lemma}[\cite{Mil16}, Corollary 5.4]\label{milnor_5.4_lemma}
    If $c_{\lambda+1} = c_{\lambda-1}=0$, then $b_\lambda = c_\lambda$. ($b_{\lambda+1} = b_{\lambda-1} = 0$ also follows.) 
\end{lemma}

\subsection{Resulting constraints on tensors eigenvalues}

We carry over the notion of nondegeneracy to eigenvectors of tensors as follows.

\begin{definition}
Let $v$ be an eigenvector which is a nondegenerate critical point of $f \colon S^{n-1} \rightarrow \R$. We define the \textbf{index} of $v$ to be its index as a nondegenerate critical point.
\end{definition}

\begin{definition}
A tensor $T$ is called \textbf{nondegenerate} if all its eigenvectors are nondegenerate.
\end{definition}

Now we are ready to obtain constraints on the eigenpairs of symmetric tensors.

\begin{theorem}
A symmetric, nondegenerate tensor has at least one eigenvector of index $n-1$.
\end{theorem}

\begin{proof}
For $M = S^{n-1}$, $b_{n-1} = 1 \leq c_{n_1}$ by \cref{weak_morse_inequality_lemma}.
\end{proof}

\begin{theorem}
For a nondegenerate, symmetric $T \in (\R^n)^{\otimes d}$, let $n_\lambda$ be the number of eigenvalues with index $\lambda$.

\begin{equation}
    \sum (-1)^\lambda n_\lambda = 
    \begin{cases}
    0 & \text{ if $n$ is even}, \\
    2 & \text{ if $n$ is odd}.
    \end{cases}
\end{equation}
\end{theorem}

\begin{proof}
The Euler characteristic obeys

\begin{equation}
    \chi(M) := \sum (-1)^\lambda b_\lambda \stackrel{\text{(1)}}{=} \sum (-1)^{\lambda} a_\lambda \stackrel{\text{(2)}}{=} \sum (-1)^\lambda c_\lambda
\end{equation}

\noindent where equality (1) comes from algebraic topology and equality (2) follows from \cref{milnor_bijection_lemma}. In our case, the manifold is $M = S^{n-1}$. For the sphere,

\begin{equation}
    H_k(S^n) = 
    \begin{cases}
    \mathbb{Z} & k \in \{0,n\}, \\
    \{0\}          & k \not\in \{0, n\}.
    \end{cases}
\end{equation}

Hence, $b_0 = b_{n-1} = 1$ and $b_\lambda = 0$ for $\lambda \not\in \{0, n-1\}$. If $n$ is odd, $\sum (-1)^\lambda b_\lambda = 2$, otherwise the sum is $0$. By noting that $n_\lambda = c_\lambda$, the result follows.

\end{proof}

\begin{theorem}
Let $T \in (\mathbb{R}^n)^{\otimes d}$ be a symmetric, nondegenerate tensor.

\begin{enumerate}[(i)]
    \item If $c_{0} \neq 1$, then $c_{1} > 0$.
    \item If, for $\lambda \in \{2, 3, \cdots, n-2\}$, $c_\lambda > 0$, then at least one of $c_{\lambda-1},c_{\lambda+1}$ is $> 0$.
    \item If $c_n > 0$, then $c_{n-1} > 0$.
    \item If $c_{n-1} \neq 1$, then at least of one $c_{n-2}, c_n$ is $>0$.
\end{enumerate}
\end{theorem}

\begin{proof}
We use the contrapositive of \cref{milnor_5.4_lemma}, which says that if $b_\lambda \neq c_\lambda$, then at least one of $c_{\lambda+1}, c_{\lambda-1}$ is nonzero. Each item follows since

\begin{enumerate}[(i)]
    \item $b_0=1$ and $c_{-1}$ cannot be nonzero since it does not exist.
    \item If $\lambda \in \{2, 3, \cdots, n-2\}$, $b_\lambda = 0$.
    \item $b_n=0$ and $c_{n+1}$ cannot be nonzero since it does not exist.
    \item $b_{n-1} = 1$.
\end{enumerate}
\end{proof}

\printbibliography

\appendix

\section{Lemmas for variational definitions}\label{appendix_var_lemmas}

Here we prove some lemmas that are used in the definitions of tensor singular tuples and eigenpairs.

\begin{lemma}[\cite{Lim05}]\label{grad_norm_lemma}
Let $1 < p < \infty$ and $\mathbb{R}^n \ni \pmb{x} \neq \pmb{0}$. Then

\begin{equation}
    \nabla \norm{\pmb{x}}_p = \frac{\varphi_{p-1}(\pmb{x})}{\norm{\pmb{x}}_p^{p-1}}
\end{equation}

\noindent where

\begin{align}
\sgn(x) &:=
\begin{cases}
+1 & \text{ if } x>0 \\
0 & \text{ if } x=0 \\
-1 & \text{ if } x<0
\end{cases} \\
\varphi_p(\pmb{x}) &:= [\sgn(x_1)|x_1|^p, \cdots, \sgn(x_n)|x_n|^p]^\top
\end{align}
\end{lemma}

Observe that here we correct $\varphi$ to include the absolute values.

\begin{proof}
Note that

\begin{align}
    \frac{\partial}{\partial x_j}\norm{\pmb{x}}_p &= \frac{\partial}{\partial x_j} \Big( \sum_{i=1}^n |x_i|^p \Big)^{1/p} \\
    &= \frac{1}{p} \Big( \sum_{i=1}^n |x_i|^p \Big)^{1/p - 1} \frac{\partial}{\partial x_j} |x_j|^p
\end{align}

Since $\frac{\partial}{\partial x_j} |x_j|^p = p|x_j|^{p-1} \sgn(x_j)$,

\begin{equation}
    \frac{\partial}{\partial x_j}\norm{\pmb{x}}_p = \norm{\pmb{x}}_p^{1-p} |x_j|^{p-1} \sgn(x_j)
\end{equation}

\noindent concluding the proof.
\end{proof}

\begin{lemma}\label{grad_f_T_lemma}
Let $f_T(\pmb{x}_1,\cdots,\pmb{x}_k)$ be the polynomial associated with the tensor $T$. If $\norm{\pmb{x}_i}=1$ for all $i$, then

\begin{equation}\label{grad_f_T_eq}
    \frac{\pmb{x}_i^\top \nabla_{\pmb{x}_i} f_T (\pmb{x}_1,\cdots,\pmb{x}_k) }{\pmb{x}_i^\top \nabla_{\pmb{x}_i} \norm{\pmb{x}_i}_p} = f_T(\pmb{x}_1,\cdots,\pmb{x}_k)
\end{equation}
\end{lemma}

\begin{proof}
First, note that

\begin{align}
\pmb{x}^\top \varphi_{p-1}(\pmb{x}) &= \sum_{j} x_j \sgn(s_j) |x_j|^{p-1} \\
&= \sum_j |x_j| \sgn^2(x_j) |x_j|^{p-1} \\
&= \sum_j |x_j|^p
\end{align}

Hence, using \cref{grad_norm_lemma}, the denominator of the left side of \cref{grad_f_T_eq} is $\norm{\pmb{x}_i}_p^{p} = 1$. Recall the definition of $f_T$ in \cref{associated_poly_eq}. It follows that

\begin{align}
\nabla_{\pmb{x}_i} f_T(\pmb{x}_1,\cdots,\pmb{x}_k) &= f\Bigg(\pmb{x}_1,\cdots,\pmb{x}_{i-1},
\begin{bmatrix}
1 \\ \vdots \\ 1
\end{bmatrix}
,\pmb{x}_{i+1},\cdots,\pmb{x}_k \Bigg) \\
&= \begin{bmatrix}
\sum_{i_m, m\neq i} T_{i_1\cdots i_{i-1}, 1, i_{i+1} \cdots i_k} \prod_{r\neq i} x_{i_r}^{(r)} \\ \vdots \\
\sum_{i_m, m\neq i} T_{i_1\cdots i_{i-1}, n_i, i_{i+1} \cdots i_k} \prod_{r\neq i} x_{i_r}^{(r)}
\end{bmatrix}
\end{align}

Hence

\begin{align}
\pmb{x}_i^\top \nabla_{\pmb{x}_i} f_T (\pmb{x}_1,\cdots,\pmb{x}_k) &= \sum_{j} x_j^{(i)} \sum_{i_m, m\neq i} T_{i_1\cdots i_{i-1}, j, i_{i+1} \cdots i_k} \prod_{r\neq i} x_{i_r}^{(r)} \\
&= f_T (\pmb{x}_1,\cdots,\pmb{x}_k)
\end{align}

\noindent proving the lemma.

\end{proof}

\begin{lemma}[Euler's theorem for homogeneous functions (\cite{Tu11}, Problem 9.6)]\label{euler_homo_lemma}
Let $f \colon \mathbb{R}^n \setminus \{0\} \rightarrow \mathbb{R}$ be continuously differentiable and homogeneous of order $k$. Then

\begin{equation}
\pmb{x}^{\top} \nabla_{\pmb{x}} f(\pmb{x}) = k f(\pmb{x})
\end{equation}
\end{lemma}

\end{document}